\documentclass[12pt,reqno,twoside,makeidx]{amsart}
\usepackage[margin=2.7cm]{geometry}
\usepackage[utf8]{inputenc}
\usepackage{mathtools}
\usepackage{graphicx}
\usepackage{amsmath,amssymb,amsfonts ,amsthm}
\usepackage{eurosym,datatool,longtable}
\usepackage{pgf}
\usepackage{tikz-cd}
\usepackage[shortlabels]{enumitem}
\usepackage[all]{xy}
\usepackage{cancel}
\usepackage{xfrac}
    \DeclareFontFamily{U}{wncy}{}
    \DeclareFontShape{U}{wncy}{m}{n}{<->wncyr10}{}
    \DeclareSymbolFont{mcy}{U}{wncy}{m}{n}
    \DeclareMathSymbol{\Sh}{\mathord}{mcy}{"58} 
\usetikzlibrary{matrix,arrows,decorations.pathmorphing}
\author[]{Simone Maletto}
\title[]{Galois representations with big image in the general symplectic group $\op{GSp}_4(\Z_p)$}
\date{}
\newtheorem{defi}{Definition}[section]
\theoremstyle{definition}
\theoremstyle{definition}

\theoremstyle{definition}
\newtheorem{cor}[defi]{Corollary}
\theoremstyle{definition}

\theoremstyle{definition}
\newtheorem{rmk}[defi]{Remark}
\theoremstyle{definition}
\newtheorem{ex}[defi]{Example}
\theoremstyle{definition}\usepackage{amsmath}
\usepackage{amsfonts}
\usepackage{amssymb}
\usepackage{graphicx}
\usepackage{amsmath,amssymb, amsthm}
\usepackage{listings}
\usepackage{color}

\newtheorem{question}[defi]{Question}
\theoremstyle{plain}
\newtheorem{thm}[defi]{Theorem}
\theoremstyle{plain}
\newtheorem{prop}[defi]{Proposition}
\theoremstyle{plain}
\newtheorem{lem}[defi]{Lemma}

\newcommand{\Ker}{\operatorname{Ker}}

\newcommand{\GL}{\operatorname{GL}}

\newcommand{\N}{\mathbb{N}}
\newcommand{\Z}{\mathbb{Z}}
\newcommand{\Q}{\mathbb{Q}}
\newcommand{\R}{\mathbb{R}}
\newcommand{\C}{\mathbb{C}}

\newcommand{\F}{\mathbb{F}}

\newcommand{\Fp}{\mathbb{F}_p}

\newcommand{\pair}{(\alpha,\beta)}
\newcommand{\Symp}{\operatorname{GSp}_4(\mathbb{Z}_p)}
\definecolor{mygreen}{RGB}{28,172,0}
\definecolor{mylilas}{RGB}{170,55,241}

\newcommand{\op}[1]{\operatorname{#1}}

\begin{document}

\maketitle
\begin{abstract}
   Let $p$ be an odd prime and $e_p$ be its irregularity index. If $4e_p+8 <\frac{p-1}{2}$ we construct a  Galois representation with image in the diagonal torus of $\op{GSp}_4(\Fp)$  that lifts to a characteristic $0$ representation unramified at all primes $\ell\neq p$ with image containing a finite index subgroup of $\Symp$.
  
\end{abstract}

\section{Introduction}

\par It is well known that elliptic curves and abelian varieties defined over the rationals give rise to compatible families of Galois representations. The study of Galois representations arising in geometry has led to exciting developments in number theory, like the Modularity theorem \cite{wiles1995modular,taylor1995ring, breuil2001modularity} and the resolution of the Sato-Tate conjecture \cite{barnet2011family}. Principally polarized abelian surfaces give rise to Galois representations with image in $\op{GSp}_4(\Z_p)$. These representations are ramified at $p$ and the set of primes at which the abelian variety has bad reduction. In this paper, we are interested in the following question.
\begin{question}
Given a prime number $p$, does there exist a continuous Galois representation
\[\rho:\op{Gal}(\bar{\Q}/\Q)\rightarrow \op{GSp}_4(\Z_p)\]
satisfying the following properties
\begin{enumerate}
    \item $\rho$ is unramified at all primes $\ell\neq p$, 
    \item the image of $\rho$ contains a finite index subgroup of $\op{Sp}_4(\Z_p)$?
\end{enumerate}
\end{question}
A related question was first studied by R.Greenberg in \cite{greenberg2016galois}, who showed that if $p$ is a regular prime such that $p\geq 4\lfloor \frac{n}{2}\rfloor +1$, then there is a continuous Galois representation $\rho:\op{Gal}(\bar{\Q}/\Q)\rightarrow \op{GL}_n(\Z_p)$ which is unramified away from $p$ and has big image. In \cite{cornut2018generators}, C.Cornut and J.Ray  further generalized Greenberg's results for groups other than $\op{GL}_n$, without relaxing the regularity assumptions. S. Tang \cite{tangalgebraic} proved that the regularity assumption can be relaxed provided the representation is allowed to ramify at a large set of primes. A.Ray \cite{ray2021constructing} proved that the results of Greenberg generalize to irregular primes as well, provided $p\geq 2^{n+2+2e_p}+3$, where $e_p$ is the irregularity index of $p$ as defined in \ref{dep}. C.Maire subsequently proved stronger results in \cite{maire2021galois}. Maire showed that for every prime number $p\geq 3$ and every integer $n\geq 1$, there exists a continuous Galois representation $\rho:\op{Gal}(\bar{\Q}/\Q)\rightarrow \op{GL}_n(\Z_p)$ which has open image
and is unramified outside $p$ (resp. outside $\{2,p\}$) when $p \equiv 3 \mod 4$ (resp. $p \equiv 1
\mod 4$).

\par The method used in this paper generalizes that of A.Ray to $\op{GSp}_4$, and shows that small primes congruent to $3$ modulo $4$ can be dealt with as well when the group is fixed. Using some group cohomology results we show that we can start from a residual diagonal representation and deform it, under certain hypotheses, producing a characteristic $0$ lift in such a way that its image contains the kernel of the quotient map $\Symp\rightarrow\op{GSp}_4(\Z/p^2\Z)$. Later, we show that if the prime $p$ of irregularity index $e_p$ is such that the inequality $4e_p+8<\frac{p-1}{2}$ holds, there always is at least one residual representation that can be lifted in such a way. Lastly we give an example of how to find such representation in the first nontrivial case of $p=37$.

We approach the lifting problem with the same strategy of \cite{ray2021constructing}, by lifting the Galois representation one step at the time. We deal with the cohomological obstruction to this process by appropriately choosing a suitable diagonal residual representation. The second issue is showing that this procedure yields a representation with big image. In order to guarantee this we choose an appropriate deformation at every step.

Noticeably we are able to produce representations with larger image as we are working in fixed dimension. We get a very precise conditions on the irregularity in Section 5, showing that there exists a residual representation suitable for the lifting procedure whenever the prime $p$ is such that the inequality $4e_p+8<\frac{p-1}{2}$ where $e_p$ is the irregularity index of $p$ .

The paper is brief and self contained, Section \ref{seccohom} contains some preliminary results on Galois cohomology that are crucial for the discussion in the two following sections where we introduce the deformation problem. Section \ref{secdefo} contains the preparatory results for the deformation problem, and the main theorem is proved in Section \ref{secthm}. Section \ref{seccount} reduces the hypothesis of the main theorem to a combinatorics problem in order to show that there are in fact diagonal residual representations that work as initial data for the deformation procedure.
\section*{Acknowledgments} The author would like to thank Vinayak Vatsal and Antonio Alfieri for the many helpful discussions.
\section{Preliminaries on Galois Cohomology}\label{seccohom}
\par In this section, we recall certain basic results in Galois cohomology, which will be integral to the lifting arguments in this paper. The standard references are \cite{ neukirch2013cohomology, serre2013galois}. For any field $K$, let $K^{\op{sep}}$ be a choice of separable closure, and let $\op{G}_K:=\op{Gal}(\bar{K}/K)$, the absolute Galois group of $K$. For each prime $\ell$, fix an embedding $\iota_\ell:\bar{\Q}\hookrightarrow \bar{\Q}_\ell$, and let  $\op{G}_\ell$ denote $\op{G}_{\Q_\ell}$. Note that the choice of embedding $\iota_\ell$ gives rise to an inclusion $\op{G}_\ell\hookrightarrow \op{G}_{\Q}$. For a finite set of primes $S$, set $\Q_S$ to be the maximal algebraic extension of $\Q$ in which the primes $\ell\notin S$ are unramified. Write $\op{G}_S$ for the Galois group $\op{Gal}(\Q_S/\Q)$ and analogously set $\op{G}_\infty=\op{Gal}(\C/\R)$. 
\par Throughout this section, we fix a prime $p$ and $S$ a finite set of primes containing $p$. Let $M$ be a finite dimensional $\F_p$-vector space on which $\op{G}_S$ acts. Note that for every prime $\ell$, $M$ is viewed as a $\op{G}_\ell$-module via the inclusion $\op{G}_\ell\hookrightarrow \op{G}_{\Q}$ . Let $\mu_p$ denote the $\F_p$-vector space generated by the group of $p$-th roots of unity and $M^*:=\op{Hom}(M, \mu_p)$. Associated to any $\F_p$-vector space is the dual space $V^{\vee}:=\op{Hom}(V, \F_p)$.

\begin{thm}[Tate Local duality] \cite[Pp. 91,92]{serre2013galois} 
Let $M$ be as above and $\ell$ be any prime number. Then for $i=0,1,2$, the cup product induces a perfect pairing
\[H^i(\op{G}_\ell,M)\times H^{2-i}(\op{G}_\ell,M^*)\xrightarrow{\cup} H^2(\op{G}_\ell,\mu_p)\xrightarrow{\sim} \F_p.\]
In particular, this pairing sets up a natural isomorphism \[H^{i}(\op{G}_\ell,M)\simeq H^{2-i}(\op{G}_\ell,M^*)^{\vee}\] for $i=0,1,2$.
\end{thm}

\begin{thm}[Tate's local Euler-characteristic formula]\cite[Theorem 7.3.1]{neukirch2013cohomology}
Let $\ell$ be a prime number, and set $h^i(\op{G}_\ell, M):=\dim_{\F_p} H^i(\op{G}_\ell, M)$. Then, we have that
\[h^0(\op{G}_\ell, M)-h^1(\op{G}_\ell, M)+h^2(\op{G}_\ell, M)=\begin{cases}
-\dim_{\F_p}M &\text{ if }\ell=p,\\
0 &\text{ if }\ell\neq p.
\end{cases}\]
\end{thm}

Let $G$ be any group and $H$ a normal subgroup of $G$. Let $M$ be a $G$-module. The inclusion map $H\hookrightarrow G$ induces an homomorphism in cohomology, called the restriction (to $H$) homomorphism \[\text{Res}:H^{i}(G,M)\rightarrow H^{i}(H,M)^{G/H}.\] Obtained by composition at the level of cochain complexes.
Similarly, as $H$ is normal we can consider the quotient $G/H$ and the $G/H$-module $M^H$ of $H$-invariants. The projection map induces again an homomorphism in cohomology, now called the inflation homomorphism, \[\text{Inf}:H^{i}(G/H,M^H)\rightarrow H^{i}(G,M).\]
\begin{prop}[Inflation-Restriction sequence]\cite[Chapter 4, Section 5]{cassels1968algebraic} 
With respect to the notation above the sequence 
\[0\rightarrow H^1(G/H,M)\xrightarrow[]{\text{Inf}}H^1(G,M)\xrightarrow[]{\text{Res}}H^1(H,M)^{G/H}\]
is exact.

Furthermore, if $H^1(H,M)=0$ the sequence
\[0\rightarrow H^2(G/H,M)\xrightarrow[]{\text{Inf}}H^2(G,M)\xrightarrow[]{\text{Res}}H^2(H,M)^{G/H}\]
is exact.
\end{prop}

\begin{thm}[Global Euler- Poincarè characteristic formula]\cite[Chapter 4, Proposition 11]{cornell2013modular}
Let $S$ be a finite set of non-archimedean primes and $M$ be any \textit{finite} $\op{G}_S$-module. Then, the groups $H^i(\op{G}_S, M)$ $i=0,1,2$ are finite and
\[\frac{\# H^0(\op{G}_S,M)\# H^2(\op{G}_S,M)}{\# H^1(\op{G}_S,M)}=\frac{\# H^0(\op{G}_\infty,M)}{\# M}.\]
\end{thm}\noindent Note that in our applications, $S=\{p\}$.

We need one more element before beginning our construction, and we shall start by giving the following
\begin{defi}
For a $\op{G}_{\{p\}}$-module $M$ we define the $i$-th $\Sh$-group unramified outside $p$ as
\[\Sh^i_{\{p\}}(M):=\Ker\left(H^i(\op{G}_{\{p\}},M))\xrightarrow{Res}H^i(G_{\Q_p},M)\right) \]
\end{defi}

\begin{thm}[Global duality for $\Sh$-groups]\cite[Theorem 8.6.7]{neukirch2013cohomology}
The Cassels-Tate pairing 
\[\Sh^2_{\{p\}}(M)\times\Sh^1_{\{p\}}(M^*)\rightarrow \Fp\] 
identifies $\Sh^2_{\{p\}}(M)$ with $\Sh^1_{\{p\}}(M^*)^\vee$.
\end{thm}

\section{The Deformation Problem for $\Symp$}\label{secdefo}
This section follows \cite{ray2021constructing} with the required modifications coming from working with the group $\op{GSp}_4$ instead of $\GL_n$. Recall that $\op{GSp}_4$ is defined as the subgroup of $\GL_4$ of matrices satisfying, for some $\lambda\in\mathbb{G}_m$
\[A^tJA=\lambda J\]
where 
\[J = \begin{pmatrix}0 & \text{Id}_2 \\ -\text{Id}_2 &0\end{pmatrix}.\]

Furthermore, the assignation $A\mapsto\nu(A)=\lambda$  defines a character, called the similitude character, $\nu:\op{GSp}_4\rightarrow \mathbb{G}_\text{m}$.
Let us now fix an odd prime $p$, and recall that for each prime $\ell$ we have an inclusion $\op{G}_{\ell}\hookrightarrow\op{G}_\Q$ of Galois groups and set $\chi$ to be the $p$-adic cyclotomic character defined as the inverse limit on $n$ of the Galois action on the $p^n$-th roots of unity.  We will denote by $\bar{\chi}$ its modulo $p$ reduction. For $m\geq 1$ denote with $\mathcal{U}_m\subset \Symp$ the kernel of the modulo-$p^m$ reduction map. Fix a pair \[\pair\in \Z/(p-1)\Z\times \Z/(p-1)\Z\] and set \[\bar{\rho}=\bar{\rho}_{\pair}:\op{G}_{\{p\}}\rightarrow \op{GSp}_4(\F_p)\] to be the following mod-$p$ Galois representation
\begin{equation}
\bar{\rho}=\begin{pmatrix} \bar{\chi}^\alpha & & & \\ & \bar{\chi}^\beta & & \\ & & \bar{\chi}^{-\alpha} & \\ & & & \bar{\chi}^{-\beta} \end{pmatrix}.
\end{equation} 

In this section, it is shown that if $\pair$ satisfies some explicit conditions, then $\bar{\rho}$ lifts to a characteristic zero Galois representation which is unramified at all primes $\ell\neq p$, and has big image. In fact, it is shown that the image will contain the group $\mathcal{U}_2$. First, we review some basic notions from the deformation theory of Galois representations. Note that in this paper, we do not study Galois deformation rings, and only consider the lifting problem. As a result, these constructions apply to reducible representations such as $\bar{\rho}$.

For a local ring $(R,\mathfrak{m}_R)$, let $\widehat{\op{GSp}_4}(R)$ be the group
\[\widehat{\op{GSp}_4}(R)=\ker\Big\{ \op{GSp}_4(R)\xrightarrow[]{\text{mod }\mathfrak{m}_R}\op{GSp}_4(R/\mathfrak{m}_R) \Big\}\].

\begin{defi}
Let $m$ be a positive integer. A mod-$p^m$ lift of $\bar{\rho}$ is a continuous homomorphism $\rho_m:\op{G}_\Q\rightarrow\op{GSp}_4(\Z/p^m\Z)$ such that $\rho_m=\bar{\rho}$ modulo $p$. Two lifts $\rho_m,\rho_m'$ are called \text{strictly equivalent} if there exists a matrix $A \in \widehat{\op{GSp}_4}(\Z/p^m\Z)$ such that $\rho_m'=A\rho_mA^{-1}$. A \textit{deformation} is a strict equivalence class of lifts.
\end{defi}

Let $\nu(\bar{\rho})$ be the similitude character of $\bar{\rho}$, meaning that $\bar{\rho}(\sigma)^tJ\bar{\rho}(\sigma)=\nu(\bar{\rho}(\sigma))J$. Let $D$ be the ring of dual numbers, i.e. $D=\F_p[\epsilon]/(\epsilon^2)$. The map $\mathfrak{gsp}_{4/\F_p}\rightarrow \op{GSp}_4(D)$ given by $A\mapsto 1+\epsilon A$ induces an isomorphism between $\mathfrak{gsp}_{4/\F_p}$ and the kernel of the mod-$\epsilon$ reduction map 
\[\op{GSp}_4\left(D\right)\rightarrow \op{GSp}_4(\F_p). \]

Let $\omega$ be the additive character on $\mathfrak{gsp}_4$ defined by the relation \[\nu(\op{Id}+\epsilon A)=1+\varepsilon \omega(A).\]  
Note that the kernel of $\omega$ is $\mathfrak{sp}_4$, the Lie algebra of $\op{Sp}_{4/\F_p}$. Let now $\tilde{\nu}:\op{GSp}_4\rightarrow\Z_p$ be the Teichmüller lift to characteristic $0$ of the similitude character $\nu$.

For any character $\phi:\op{G}_\Q\rightarrow\Z^\times_p$, we denote with $\phi_m$ its mod-$p^m$ reduction. Let us now fix a character $\psi:\op{G}_\Q\rightarrow\Z^\times_p$ that is unramified outside $\{p\}$ and is equal to $\tilde{\nu}(\rho)$ modulo $p^2$. In the notation just introduced we can write \[\psi_2=\tilde{\nu}(\rho)_2 .\] 
From now on we will consider only deformations of $\bar{\rho}$ such that $\nu(\rho)=\psi$. 
We can now introduce the deformation problem.
\begin{defi}\label{d:3}
Set $\op{Ad}\bar{\rho}$ to denote the Galois module over $\Fp$ defined by the composition of $\bar{\rho}$ with the adjoint representation of $\op{GSp}_4$, and let $\op{Ad}^0\bar{\rho}$ be the Galois sub-module given by the kernel of $\omega$ equipped with the action obtained by composition of $\bar{\rho}$ and the adjoint action. 
\end{defi}
More explicitly, given $X\in \op{Ad}\bar{\rho}$ and $g\in \op{G}_{\{p\}}$, we have that
\[g\cdot X:=\bar{\rho}(g) X \bar{\rho}(g)^{-1}.\]
The Galois module $\op{Ad}\bar{\rho}$ is equipped with a Lie bracket $[X,Y]\coloneqq XY-YX$. The underlying vector space of Ad $\bar\rho$ (resp. $\op{Ad}^0\bar{\rho}$) is the Lie algebra $\mathfrak{gsp}_4$ of $\op{GSp}_{4/\Fp}$ (resp. $\op{Sp}_{4/\Fp})$. Such Lie algebra is defined by the linearization of the equation defining $\op{GSp}_4$, which is the the algebraic condition $B^tJ+JB=\omega(B)I$ (resp.$B^tJ+JB=0$ ). An easy calculation then shows that this is the space of matrices of the form (in 2-by-2 blocks)
\[\begin{pmatrix}A & B\\ C & -A\end{pmatrix}\]
where $B$ and $C$ are symmetric. Now we are ready to introduce some notation that will be used throughout to refer to the structure of Ad$^0\bar{\rho}$.

Set $D:=\{(\delta_1,\delta_2)\in\Z^2\text{ such that } |\delta_1|+|\delta_2|=2\}$, and let $\mathfrak{t}$ be the two-dimensional sub-module of $\op{Ad}^0\bar{\rho}$ generated by the matrices $\op{diag}(1,0,-1,0)$ and $\op{diag}(0,1,0,-1)$. Note that the Galois action on $\mathfrak{t}$ is trivial by definition of $\bar{\rho}$. Given a character $\eta:\op{G}_{\{p\}}\rightarrow \F_p^{\times}$, denote by $\F_p(\eta)$ the one dimensional $\F_p$-vector space on which $\op{G}_{\{p\}}$ acts via $\eta$.

\begin{prop}\label{p:2}
With respect to notation above, the Galois module $\op{Ad}^0\bar{\rho}$ decomposes into a direct sum of sub-modules
\[\op{Ad}^0\bar{\rho}= \mathfrak{t}\oplus\Big(\bigoplus\limits_{(\delta_1,\delta_2)\in D}\Fp(\bar{\chi}^{\delta_1\alpha+\delta_2\beta})\Big).\]
We will denote with $X_{\delta_1,\delta_2}$ the generator of the $(\delta_1,\delta_2)$-component on the right-hand side.
\end{prop}

\begin{proof}
By Definition \ref{d:3} the Galois action on Ad$^{0}\bar{\rho}$ is given by the composition of $\bar{\rho}$ and the adjoint action. Recall by (1) that $\bar{\rho}(\sigma)$ is given by the diagonal matrix 
\[\bar{\rho}(\sigma)=\begin{pmatrix} \bar{\chi}^\alpha(\sigma) & & & \\ & \bar{\chi}^\beta(\sigma) & & \\ & & \bar{\chi}^{-\alpha}(\sigma) & \\ & & & \bar{\chi}^{-\beta}(\sigma) \end{pmatrix}.\]In particular, for any $\sigma\in\op{G}_{\{p\}}$, the action on a matrix $M$ in the Lie algebra will be given by $\bar{\rho}(\sigma)M\bar{\rho}(\sigma)^{-1}$, this can be written explicitly as
\[{\bar{\rho}(\sigma)\begin{pmatrix} a_{1,1} & a_{1,2} & a_{1,3} & a_{1,4} \\ a_{2,1} & a_{2,2} & a_{1,4} & a_{2,4} \\ a_{3,1} & a_{3,2} & -a_{1,1} & -a_{2,1} \\ a_{3,2} & a_{4,2} & -a_{1,2} & -a_{2,2} \end{pmatrix} \bar{\rho}(\sigma)^{-1}}\]
computing the product we get the matrix
\[\begin{pmatrix} a_{1,1} &\bar{\chi}^{\alpha-\beta}(\sigma) a_{1,2} & \bar{\chi}^{2\alpha}(\sigma)a_{1,3} & \bar{\chi}^{\alpha+\beta}(\sigma)a_{1,4} \\ \bar{\chi}^{\beta-\alpha}(\sigma)a_{2,1} & a_{2,2} & \bar{\chi}^{\beta+\alpha}(\sigma)a_{1,4} & \bar{\chi}^{2\beta}(\sigma)a_{2,4} \\ \bar{\chi}^{-2\alpha}(\sigma)a_{3,1} & \bar{\chi}^{-\alpha-\beta}(\sigma)a_{3,2} & -a_{1,1} & -\bar{\chi}^{\beta-\alpha}(\sigma)a_{2,1} \\ \bar{\chi}^{-\alpha-\beta}(\sigma)a_{3,2} & \bar{\chi}^{-2\beta}(\sigma)a_{4,2} & -\bar{\chi}^{\alpha-\beta}(\sigma)a_{1,2} & -a_{2,2} \end{pmatrix}\]
\end{proof}
\begin{defi}
A generator of the eigenspace $\Fp(\bar{\chi}^{\delta_1\alpha+\delta_2\beta})$ as in the proposition above is called a root vector and from now on we will denote it with $X_{\delta_1,\delta_2}$.
\end{defi}

Let once again $m$ be a positive integer and let $\rho_m$ be a deformation of $\bar{\rho}$. We want to understand if $\bar{\rho}$ lifts to a continuous representation in characteristic $0$ with large image (meaning its image contains the kernel of the reduction modulo $p^n$ for some $n$). The obstruction to the existence of such lift will be encoded into a $2$-cocycle. 
To define such we will need the following 
\begin{lem}
Let $\rho_m:\op{G}_{\{p\}}\rightarrow\op{GSp}_4(\Z/p^m\Z)$ be a continuous representation lifting $\bar{\rho}$ to characteristic $p^m$ and $\psi:\op{G}_{\{p\}}\rightarrow\Z^\times_p$ such that $\psi_{m}(\sigma)=\nu_m(\rho_m(\sigma))$. Then there exists a continuous set-theoretic lift $ \varrho:\op{G}_{\{p\}}\rightarrow\op{GSp}_4(\Z/p^{m+1}\Z) $ (which we do not require to be a group homomorphism) such that $\nu_{m+1}(\varrho(\sigma))=\psi_{m+1}(\sigma)$
\end{lem}
\begin{proof}
Let $\sigma\in\operatorname{G}_{\{p\}}$, we need to find an element $\varrho(\sigma)\in\operatorname{GSp}_4(\Z/p^{m+1}\Z)$ such that
\begin{enumerate}
    \item$\varrho(\sigma)\equiv\rho_{m}(\sigma)$ modulo $p^m$
    \item $\nu_{m+1}(\varrho(\sigma))=\psi_{m+1}(\sigma)$
\end{enumerate}
To do so, we take an element $R_\sigma\in\op{GSp}_4(\Z/p^{m+1}\Z)$ that reduces to $\rho_{m}(\sigma)$ modulo $p^m$ and consider its orbit under multiplication by the matrix
\[A=\begin{pmatrix}1+p^m&&&\\&1&&\\&&1&\\&&&1+p^m\end{pmatrix}\in\op{GSp}_4(\Z/p^{m+1}\Z).\]
Each element in the orbit of $R_\sigma$ under the action of $A$ is reduces to $\rho_{m}(\sigma)$ modulo $p^m$ as $A\equiv 1$ modulo $p^m$.
Furthermore, $\nu_{m+1}(A)=1+p^{m}\in(\Z/p^{m+1}\Z)^\times$ and generates the kernel of the projection $(\Z/p^{m+1}\Z)^\times\rightarrow(\Z/p^m\Z)^\times$. Since $\psi_m(\sigma)$ is the reduction modulo $p^m$ of $\psi_{m+1}(\sigma)$ one has  \[\nu_{m+1}(R_\sigma)\psi_{m+1}(\sigma)^{-1}\in\ker((\Z/p^{m+1}\Z)^\times)\rightarrow(\Z/p^m\Z)^\times)=\langle\nu_{m+1}(A)\rangle.\] Therefore there is an integer $s(\sigma)$ (unique modulo $p$) such that $A^{s(\sigma)}R_\sigma $ satisfies conditions (1) and (2). Hence the map
\[\varrho(\sigma)\coloneqq A^{s(\sigma)}R_\sigma.\]
is a continuous set-theoretic lift.
\end{proof}


 We now identify $\op{Ad}^0\bar{\rho}$ with the kernel of the reduction ${\op{Sp}_4(\Z/p^{m+1}\Z)\rightarrow\op{Sp}_4(\Z/p^m\Z)}$ by associating to an element $1+p^m\tilde{M}$ the matrix $M\in\op{Ad}^0\bar{\rho}$ so that $\tilde{M}=M$ modulo $p$. Let $\mathcal{O}(\rho_m)$ be the cohomology class in $H^2(\op{G}_{\{p\}},\op{Ad}^0\bar{\rho})$ defined by the $2$-cocycle
\[(g,h)\mapsto \varrho(gh)\varrho(h)^{-1}\varrho(g)^{-1}.\]
The associated cohomology class, $\mathcal{O}(\rho_m)$ is independent of the choice of lift $\varrho$, as given two lifts $\varrho_1,\varrho_2$ of $\rho_m$ the cohomology classes these define differ by the $2$-coboundary $d(\varrho^{-1}_2\varrho_1)$. We have the following
\begin{lem}
A mod-$p^m$ deformation $\rho_m$ lifts one more step to a Galois representation $\rho_{m+1}$ which is unramified outside $\{p\}$ if and only if $\mathcal{O}(\rho_m)=0$.
\end{lem}
\begin{proof}
Clearly, if $\rho_m$ admits a lift $\rho_{m+1}$, as $\mathcal{O}(\rho_m)$ does not depend on the choice of (set-theoretic) lift, we can take $\varrho=\rho_{m+1}$ so that the cocycle is trivial.
Conversely, if $\mathcal{O}(\rho_m)=0$, then there exists a $1$-cochain $\phi$ such that the $2$-coboundary $d\phi$ obtained by applying the differential $d$ on $\phi$ satisfies $\mathcal{O}(\rho_m)=d\phi$ at the level of chain complexes, meaning that
\[\varrho(gh)\varrho(h)^{-1}\varrho(g)^{-1}=\phi(gh)^{-1}\phi(g)^g\phi(h).\]
Recalling that the galois action is given by conjugation by $\bar{\rho}(g)$ we get
\[\phi(gh)\varrho(gh)(\phi(g)\varrho(g)\phi(h)\varrho(h))^{-1}=1.\]
meaning that the map $\tilde{\varrho}$ defined by $\tilde{\varrho}(g)=\phi(g)\varrho(g)$ is a homomorphism. Moreover, going back to $\op{Sp}_4$, we define $\rho_{m+1}(g)=(1+p^m\widetilde{\phi(g)})\varrho(g)$, where $\widetilde{\phi(g)}$ is the $1$-cochain $\phi\in Z(\op{G}_{\{p\}},\text{Ad}^0\bar{\rho})$ composed with the isomorphism $\text{Ad}^0(\bar{\rho})\equiv \ker(\op{GSp}_4(\Z/p^{m+1}\Z))\rightarrow\op{GSp}_4(\Z/p^m\Z)$,  which clearly still lifts $\rho_{m}$.
\end{proof}

The following fact determines the structure of the set of deformations $\rho_{m+1}$ of $\rho_m$.
\begin{lem}
Suppose that there exists two deformations $\rho_{m+1},\rho'_{m+1}$ of $\rho_m$. Then there exists a unique class $h\in H^1(\op{G}_{\{p\}},\op{Ad}^0\bar{\rho})$ such that any of its modulo $p^{m+1}$ lifts $\tilde{h}$ satisfies 
\[\rho'_m=(1+p^m\tilde{h})\rho_{m+1}.\]
\end{lem}
\begin{proof}
define $1+p^m\tilde{h}:\op{G}_{\{p\}}\rightarrow \Ker\left(\op{Sp}_4(\Z/p^{m+1}\Z)\rightarrow\op{Sp}_4(\Z/p^m\Z)\right)\cong \text{Ad}^0\bar{\rho}$ as
\[1+p^m\tilde{h}(\sigma)=\rho_{m+1}(\sigma)\rho'_{m+1}(\sigma)^{-1}.\]
We check that using the isomorphism $\phi:\Ker\left(\op{Sp}_4(\Z/p^{m+1}\Z)\rightarrow\op{Sp}_4(\Z/p^m\Z)\right)\cong \text{Ad}^0\bar{\rho}$ we obtain a $1$-cocycle $h=\phi\circ(1+p^m\tilde{h})$.
\begin{equation*}
    \begin{aligned}
      h(\sigma\tau)&=\phi(\rho_{m+1}(\sigma\tau)\rho'_{m+1}(\sigma\tau)^{-1})=\phi(\rho_{m+1}(\sigma)\rho_{m+1}(\tau)\rho'_{m+1}(\tau)^{-1}\rho'_{m+1}(\sigma)^{-1})\\
      &=\phi(\rho_{m+1}(\sigma))h(\tau)\phi(\rho'_{m+1}(\sigma)^{-1})\\
      &=\phi(\rho_{m+1}(\sigma)\rho'_{m+1}(\sigma)^{-1})\phi(\rho'_{m+1}(\sigma))h(\tau)\phi(\rho'_{m+1}(\sigma)^{-1})=h(\sigma)h(\tau)^\sigma.
    \end{aligned}
\end{equation*}
Let us now consider $h,h'\in H^1(\op{G}_{\{p\}},\text{Ad}^0\bar{\rho})$, and let $\rho_{m+1}$ be a deformation of $\rho_m$. Suppose that $(1+p^m\tilde{h})\rho_{m+1}=(1+p^m\tilde{h}')\rho_{m+1}$. We have \[(1+p^m\tilde{h})\rho_{m+1}((1+p^m\tilde{h}')\rho_{m+1})^{-1}=1.\] Since we are working mod $p^{m+1}$, $(1+p^m\tilde{h}')^{-1}=1-p^m\tilde{h}'$ and the previous equality rewrites as $(1+p^m\tilde{h})(1-p^m\tilde{h}')=1+p^m(\tilde{h}-\tilde{h'})=1$. which yields $h=h'$.
\end{proof}
We say that the ``unramified outside $\{p\}$" deformation problem for $\bar{\rho}$ is \textit{unobstructed} if $H^2(\op{G}_{\{p\}},\op{Ad}^0\bar\rho)$ is trivial. From the previous discussion we automatically have the following.
\begin{lem}
Suppose that the unramified outside $\{p\}$ deformation problem for $\bar\rho$ is unobstructed and suppose we are given a deformation $\rho_m$ of $\rho$. Then there exists a deformation
\[\rho:\op{G}_{\{p\}}\rightarrow \op{GSp}_4(\Z_p)\]
such that $\rho=\rho_m$ mod $p^m$.
\end{lem}

In the next section we will describe sufficient conditions for a representation $\bar{\rho}=\bar{\rho}_{\pair}$ to be unobstructed outside $\{p\}$ in the sense defined here, and in Section $5$ we will find a condition on $p$ that guarantees the existence of $\pair$ such that the representation $\bar{\rho}_{\pair}$  is unobstructed outside $\{p\}$. Before doing so, we take some time to define some more objects that we will use to prove results on the size of the image $\op{Im}(\rho)$ of the lift we will produce.

Let $H$ be a closed subgroup of $\Symp$, the projection maps induce a filtration $\{H_i\}_{i\in\N}$ of $H$, where $H_i=H\cap \mathcal{U}_i$.
\begin{defi}
For any closed subgroup $H$ of $\Symp$ we denote $\Phi_i(H)$ the quotient $H_m/H_{m+1}$ viewed as a subset of Ad$^0\bar{\rho}$.  We will consider the case $H=\text{Im}(\rho)$ and use the notation $\Phi_i(\rho)\coloneqq\Phi_i(\text{Im}(\rho))$ for brevity.
\end{defi}

Suppose $\rho$ is a lift of $\bar{\rho}$ to characteristic $0$ and recall we write $\rho_m$ for its mod-$p^m$ reduction. Note that an element $v$ of $\Phi_m(\rho)$ can be lifted to an element of $\ker(\rho_m)/\ker(\rho_{m+1})$, which we will denote as $1+p^m\tilde{v}$. We will use the existence of such lifts to do most of the calculations in this section.



By definition each of the $\Phi_i(\rho)$ is a Galois-stable submodule of Ad$^0\bar{\rho}$. Recall that the similitude character of $\rho$ is stipulated to be equal to the character $\psi$ which is chosen to be congruent to $\tilde{\nu}$ modulo $p^2$. Therefore, for $g\in\ker \bar{\rho}$ it follows that $\nu\left(\rho_2(g)\right)=\psi_2(g)=\tilde{\nu}_2(\bar{\rho}(g))=1$. Writing $\rho_2(g)=\op{Id}+pX$, we get the equality
\[\nu(\rho_2(g))=1+p\omega (X),\]
and it follows that $\omega(X)$=0 in $\Fp$ so that $\Phi_1(\rho)\subseteq\op{Ad}^0\bar{\rho}$.

\begin{lem}\label{l:2}
Let $\rho$ be as above. For $l,m\geq 1$, we have $[\Phi_l(\rho),\Phi_m(\rho)]\subseteq \Phi_{l+m}(\rho)$ viewing each as a subset of Ad$^0\bar{\rho}$.
\end{lem}
\begin{proof}
As anticipated, in order to get this result, we will work with appropriate lifts to characteristic $p^{l+m+1}$. This approach allows us to reduce a problem on the sub-modules $\Phi_i(\rho)$ with respect to the Lie bracket to a computation on matrices. 
Let $c\in\Phi_l(\rho)$ and $d\in\Phi_m(\rho)$, 
let $C=1+p^l\tilde{c}+p^{l+1}S$ and $D=1+p^m\tilde{d}+p^{m+1}T$ in $\op{GSp}_4(\Z/p^{m+l+1}\Z)$ be lifts of $c$ and $d$ as elements of $\text{Im}(\rho_{m+l+1})$. We want to show that the element $CDC^{-1}D^{-1}$ in $\text{Im}(\rho_{m+l+1})$ is a lift of the element $[c,d]$ of Ad$^0(\bar{\rho})$ and lies in $\Phi_{m+l}(\rho)$. Assume without loss of generality that $l\leq m$. Since we are working modulo $p^{m+l+1}$ it follows that $p^m\tilde{d}p^{m+1}T=0$ and $(p^m\tilde{d})^3=0$. We have that
\[CDC^{-1}D^{-1}=(1+p^l\tilde{c}+p^{l+1}S)(1+p^m\tilde{d}+p^{m+1}T)(1+p^l\tilde{c}+p^{l+1}S)^{-1}(1+p^m\tilde{d}+p^{m+1}T)^{-1}.\]
In order to ease the form of the upcoming calculation let us rewrite this as
\[CDC^{-1}D^{-1}=(1+p^l\tilde{c}+p^{l+1}S)(p^m\tilde{d}+1+p^{m+1}T)(1+p^l\tilde{c}+p^{l+1}S)^{-1}(1+p^m\tilde{d}+p^{m+1}T)^{-1}.\]
Upon noticing that we have the identities
\begin{equation*}
\begin{aligned}
  &(1+p^m\tilde{d}+p^{m+1}T)^{-1} =(1-p^m\tilde{d}-p^{m+1}T+(p^m\tilde{d})^2),\\
  &(1+p^l\tilde{c}+p^{l+1}S)^{-1}= (1-p^l\tilde{c}-p^{l+1}S+\dots+(-1)^{m+1}(p^l\tilde{c})^{m+1}) ;\end{aligned}
\end{equation*} and unraveling the second factor and  the one above becomes
\begin{equation*}\begin{aligned}
 (1+p^l\tilde{c}+p^{l+1}S)p^m\tilde{d}(1-p^l\tilde{c}-p^{l+1}S+\dots+(-1)^{m+1}(p^l\tilde{c})^{m+1})(1-p^m\tilde{d}-p^{m+1}T+(p^m\tilde{d})^2)+ \\+(1-p^m\tilde{d}-p^{m+1}T+(p^m\tilde{d})^2)+p^{m+1}T.
\end{aligned}
\end{equation*}
After cancelling the factors whose degrees add up to more than $l+m$ we obtain
\[(1+p^l\tilde{c}+p^{l+1}S)p^m\tilde{d}(1-p^l\tilde{c}-p^{l+1}S)(1-p^m\tilde{d}-p^{m+1}T)+(1-p^m\tilde{d}-p^{m+1}T+(p^m\tilde{d})^2)+p^{m+1}T.\]
Finishing the calculations we get
\[p^m\tilde{d}+p^{m+1}\widetilde{[c,d]}-(p^m\tilde{d})^2+1-p^m\tilde{d}-p^{m+1}T+(p^m\tilde{d})^2+p^{m+1}T=1+p^{m+l}\widetilde{[c,d]} .\]
This shows $CDC^{-1}D^{-1}$ reduces to the element $[c,d]$ in $\Phi_{m+l}(\rho)$.
\end{proof}

The following Lemma will be used to show that the representation we construct contains a finite index subgroup of $\op{Sp}_4(\Z_p)$. Recall the definition of $\mathcal{U}_m$ as in the beginning of Section \ref{secdefo} is $\mathcal{U}_m\coloneqq\Ker(\Symp\rightarrow\op{GSp}_4(\Z/p^m\Z)).$
\begin{lem}\label{l:3}
Let $\rho:\op{G}_{\{p\}}\rightarrow \Symp$ be a continuous Galois representation lifting $\bar{\rho}$. Let $m\geq 1$ be such that $\Phi_m(\rho)$ contains $\op{Ad}^0\bar{\rho}$. Then we have the following:
\begin{enumerate}
    \item $\Phi_k(\rho)$ contains $\op{Ad}^0\bar{\rho}$ for $k\geq m$,
    \item the image of $\rho$ contains $\mathcal{U}_m$.
\end{enumerate}
\end{lem}
\begin{proof}
Since $[\op{Ad}^0\bar{\rho},\op{Ad}^0\bar{\rho}]=\op{Ad}^0\bar{\rho}$, condition (1) follows by the previous lemma. Let now $H$ be the image of $\rho$. Since the map is continuous and $\op{G}_{\{p\}}$ is compact $H$ is closed. For $l\geq 1$ let $H_l$ be the projection of $H$ to $\op{GSp}_4(\Z/p^l\Z).$ Since $H$ is closed we have $H=\lim\limits_{\leftarrow l}H_l$, and we are reduced to check that $\Phi_k(H)$ contains $\op{Ad}^0\bar{\rho}$ for $k\geq m$, which follows from (1). 
\end{proof}
\begin{lem}\label{lphi}
Let $\rho:\op{G}_{\{p\}}\rightarrow\Symp$ be a continuous Galois representation lifting $\bar{\rho}$. Assume that $\Phi_1(\rho)$ contains an element $w=a\op{diag}(1,0,-1,0)+b\op{diag}(0,1,0,-1)$ with $a\neq b$, furthermore assume it contains the root vectors $X_{\delta_1,\delta_2}$ for all the roots with $\delta_i=\pm1$. Then we have the following:
\begin{enumerate}
    \item $\Phi_2(\rho)$ contains $\op{Ad}^0\bar{\rho}$,
    \item The image of $\rho$ contains $\mathcal{U}_2$
\end{enumerate}
\end{lem}
\begin{proof}
By Lemma \ref{l:2}, $[\Phi_1(\rho),\Phi_1(\rho)]$ is contained in $\Phi_2(\rho)$. Using the notation above it is easy to check that the following
\begin{itemize}
    \item $[w,X_{\pm1,\pm1}]=cX_{\pm1,\pm1}$ for a nontrivial constant $c$,
    \item ${[X_{\pm1,\pm1},X_{\pm1,\mp1}]=\mp2X_{\pm2,0}}$,
    \item ${[X_{\pm1,\pm1},X_{\mp1,\pm1}]=\mp2X_{0,\pm2}}$,
    \item $[X_{1,1},X_{-1,-1}]=\text{diag}(1,1,-1,-1)\in\mathfrak{t}$,
    \item $[X_{1,-1},X_{-1,1}]$=$\text{diag}(1,-1,-1,1)\in\mathfrak{t}$.
\end{itemize}
Since the two diagonal matrices generate $\mathfrak{t}$  and we have all the other root vectors, we have  that Ad$^0(\bar{\rho})\subset\Phi_2$ and (1) is satisfied. By Lemma \ref{l:3} condition (2) holds as well, concluding the proof.
\end{proof}
\section{Deforming a suitable representation}\label{secthm}
Recall from (1) that, for a choice of $\pair$, $\bar{\rho}=\bar{\rho}_{\pair}$ is the representation.
\[\bar{\rho}=\begin{pmatrix} \bar{\chi}^\alpha & & & \\ & \bar{\chi}^\beta & & \\ & & \bar{\chi}^{-\alpha} & \\ & & & \bar{\chi}^{-\beta} \end{pmatrix}:\op{G}_{\{p\}}\rightarrow\op{GSp}_4(\Fp).\]
In this section, we show that if the pair $\pair$ satisfies certain hypothesis, then the associated representation $\bar{\rho}$ lifts to a characteristic $0$ representation with big image. The proof of the main theorem requires an extra technical step to guarantee the vanishing of the second cohomology group $H^2(\op{G}_{\{p\}},\text{Ad}^0(\bar{\rho}))$. The main idea is explained in \cite{ray2021constructing} and relies on applying the inflation-restriction sequence associated to the exact sequence of Galois groups
\[1\rightarrow\op{G}_{\Q(\mu_p)}\rightarrow\op{G}_\Q\rightarrow\operatorname{Gal}(\Q(\mu_p)/\Q)\rightarrow 1.\]
This allows A.Ray to relate properties of the groups $\Sh^i_{\{p\}}$ defined in Section 2 with properties of the class group of $\Q(\mu_p)$.
Let $A$ be the class group of $\Q(\mu_p)$ and let $\mathcal{C}$ denote the mod-$p$ class group $\mathcal{C}\coloneqq A\otimes\Fp$. The Galois group $\op{Gal}(\Q(\mu_p)/\Q)$ acts on $\mathcal{C}$ via the natural action. Since the order of $\op{Gal}(\Q(\mu_p)/\Q)$ is prime to $p$, it follows that $\mathcal{C}$ decomposes into eigenspaces
\[\mathcal{C}=\bigoplus\limits_{i=0}^{p-1}\mathcal{C}(\bar{\chi}^i),\]
where $\mathcal{C}(\bar{\chi}^i)=\{x\in\mathcal{C}| g\cdot x=\bar{\chi}^i(g)x\}$.
\begin{defi}\label{dep}
The irregularity index $e_p$ of a prime $p$ is the number of eigenspaces $C(\bar{\chi}^i)$ that are non-zero.
\end{defi}

\begin{lem}\label{lray}
For $0\leq i \leq p-2$ the following hold:
\begin{enumerate}
    \item the group $\Sh^1_{\{p\}}(\Fp(\bar{\chi}^i))$ is zero if $\mathcal{C}(\bar{\chi}^i)$ is zero
    \item the group $\Sh^2_{\{p\}}(\Fp(\bar{\chi}^i))$ is zero if $\mathcal{C}(\bar{\chi}^{p-i})$ is zero
\end{enumerate}
\end{lem}
\begin{proof}
See \cite[Lemma 3.2]{ray2021constructing}.
\end{proof}
Recall that we called $D=\{(\delta_1,\delta_2) \in \Z^2|\;\, |\delta_1|+|\delta_2|=2\}$, we have the following
\begin{thm}\label{thmdefo}
Let $\alpha,\beta$ and $\bar{\rho}$ be as above, and assume the following conditions hold:
\begin{enumerate}
    \item $\alpha+\beta$ is odd;
    \item all the $\delta_1\alpha+\delta_2\beta$ are distinct and different from 1 modulo $p-1$ \\ for { $(\delta_1,\delta_2)\in D$ };
    \item $\mathcal{C}(\bar{\chi}^{p-(\delta_1\alpha+\delta_2\beta)})=0$ for all couples $(\delta_1,\delta_2)$ in D.
\end{enumerate}
Then there exists a lift $\rho:\op{G}_{\{p\}}\rightarrow\Symp$ of $\bar{\rho}$ such that the image of $\rho$ contains $\mathcal{U}_2$.
\end{thm}
\begin{rmk}
While condition $3$ in the hypothesis above looks potentially quite daunting to check, one can immediately reduce it to checking the non-divisibility by $p$ of some Bernoulli numbers. Thanks to the Herbrand-Ribet theorem \cite{ribet1976modular} we have
\[\mathcal{C}(\bar{\chi}^{p-(\delta_1\alpha+\delta_2\beta)})=0 \iff p\not| B_{1+\delta_1\alpha+\delta_2\beta-p}. \]
\end{rmk}
\begin{proof}
First we show that the unramified outside $p$ deformation problem associated to $\bar{\rho}$ as in the hypothesis is unobstructed, i.e. $H^2(\op{G}_{\{p\}},\op{Ad}^0\bar{\rho})=\{0\} $. To see this we use the decomposition of Proposition \ref{p:2}, which tells us that the second cohomology group of $G_p$ (not $G_{\{p\}}$ yet) decomposes into
\[H^2(\op{G}_p,\op{Ad}^0\bar{\rho})=H^2(\op{G}_p,\mathfrak{t})\oplus\bigg(\bigoplus\limits_{(\delta_1,\delta_2\in D)}H^2(\op{G}_p,\Fp(\bar{\chi}^{\delta_1\alpha+\delta_2\beta}))\bigg).\]
By local duality we have that $H^2(\op{G}_p,\mathfrak{t})\cong H^0(\op{G}_p,\mathfrak{t}^*)^\vee$, which is $0$, as the Galois action on $\mathfrak{t}$ is trivial and its dual acquires a twist by $\bar{\chi}$. Similarly, again by local duality we have \[H^2(\op{G}_p,\Fp(\bar{\chi}^{\delta_1\alpha+\delta_2\beta}))\cong H^0(\op{G}_p,\Fp(\bar{\chi}^{p-\delta_1\alpha-\delta_2\beta}))^\vee.\]
Since $\pair$ was assumed to be such that $\delta_1\alpha+\delta_2\beta$ are different from $1$ for all $(\delta_1,\delta_2)\in D$, we have that $H^0(\op{G}_p,\Fp(\bar{\chi}^{p-\delta_1\alpha-\delta_2\beta}))=\{0\}$. Thus $H^2(\op{G}_p,\op{Ad}^0\bar{\rho})=0$, which yields
\[H^2(\op{G}_{\{p\}},\op{Ad}^0\bar{\rho})=\Sh^2_{\{p\}}(\op{Ad}^0\bar{\rho}).\]
By global duality, we have that
\[\Sh^2_{\{p\}}(\mathfrak{t})\cong\Sh^1_{\{p\}}(\mathfrak{t}^*)^\vee.\]
It is well known that $\mathcal{C}(\bar{\chi})=0$ \cite[Proposition 6.16]{washington1997introduction}. It follows by Lemma \ref{lray} that $\Sh^1_{\{p\}}(\Fp(\bar{\chi}))=0$ and thus $\Sh^1_{\{p\}}(\mathfrak{t^*})=0$. By assumption 3, $\mathcal{C}(\bar{\chi}^{p-\delta_1\alpha-\delta_2\beta})=0$, so that again, by Lemma \ref{lray}, $\Sh^2_{\{p\}}(\F_p(\bar{\chi}^{\delta_1\alpha+\delta_2\beta}))=0$, yielding
\[H^2(\op{G}_{\{p\}},\op{Ad}^0\bar{\rho})=\{0\}.\]

Now we seek an appropriate twist to ensure the lifted representation will halve large image. Recall that $\chi_2$ is $\chi$ modulo $p^2$, let $\rho'_2$ be the lift
\[\rho'_2=\begin{pmatrix} \chi_2^\alpha& & & \\ &\chi_2^\beta&&\\ &&\chi_2^{-\alpha}&\\ &&&\chi_2^{-\beta}\end{pmatrix}:\op{G}_{\{p\}}\rightarrow\op{GSp}_4(\Z/p^2\Z).\]
Recall $\alpha+\beta$ is assumed to be odd. Since $H^2(\op{G}_{\{p\}},\Fp(\bar{\chi}^{\pm\alpha\pm\beta}))$ and $H^0(\op{G}_\infty,\Fp(\bar{\chi}^{\pm\alpha\pm\beta}))$ are trivial, it follows from the  Global Euler characteristic formula that $H^1(\op{G}_{\{p\}},\Fp(\bar{\chi}^{\pm\alpha\pm\beta}))$ is one dimensional.
Let $f_{\pm1,\pm1}$ be a generator of $H^1(\op{G}_{\{p\}},\Fp(\bar{\chi}^{\pm\alpha\pm\beta}))$ and let $F$ be the cocycle in $H^1(\op{G}_{\{p\}},\op{Ad}^0\bar{\rho})$ obtained as the ``sum" of all the $f_{\pm1,\pm1}$. Call $\rho_2$ the twist of $\rho'_2$ by $(1+p\tilde{F})$ where $\tilde{F}$ is a lift of $F$ to characteristic $p^2$, $(1+p\tilde{F})\rho'_2:\op{G}_{\{p\}}\rightarrow\op{GSp}_4(\Z/p^2\Z)$. Since $\rho_2$ reduces to $\bar{\rho}$ modulo $p$ and $H^2(\op{G}_{\{p\}},\op{Ad}^0\bar{\rho})=0$ the deformation problem is still unobstructed and $\rho_2$ lifts to a characteristic zero Galois representation $\rho:\op{G}_{\{p\}}\rightarrow\Symp$.

In order to show that the image of $\rho$ contains $\mathcal{U}_2$, it is enough to show that $\Phi_1(\rho)$ satisfies the hypothesis of Lemma \ref{lphi}.

The image of $\bar{\rho}$ is prime to $p$ hence any Galois-submodule $M$ of Ad$^0\bar{\rho}$ decomposes as
\[M=M_1\oplus\Bigg(\bigoplus\limits_{(\delta_1,\delta_2)\in D}M_{\bar{\chi}^{\delta_1\alpha+\delta_2\beta}}\Bigg)\]
where $M_{\bar{\chi}^{\delta_1\alpha+\delta_2\beta}}$ is the submodule
\[M_{\bar{\chi}^{\delta_1\alpha+\delta_2\beta}}\coloneqq\{x\in M| g\cdot x = \bar{\chi}^{\delta_1\alpha+\delta_2\beta}(g)x \; \forall g\in\op{G}_{\{p\}}\}\]
and $M_1$ is the $\op{G}_{\{p\}}$-invariant submodule. Recall that all the $\delta_1\alpha+\delta_2\beta$ are assumed to be distinct and different from $1$ modulo $p-1$, making the decomposition above consistent, and $M_{\bar{\chi}^{\delta_1\alpha+\delta_2\beta}}$  is either $0$ or one dimensional and generated by the $X_{\delta_1,\delta_2}$ of Lemma \ref{lphi}. Since the degree of $\Q(\mu_p)/\Q$ is prime to $p$, it follows that
\[H^1\left(\op{Gal}(\Q(\mu_p)/\Q),\F_p(\bar{\chi}^{\pm\alpha,\pm\beta})\right)=0.\]
The inflation-restriction sequence now shows that the restriction of $f_{\pm1,\pm1}$ to $\op{G}_{\Q(\mu_p)}$ is non-zero. Hence, there exists $g\in\Ker\bar{\rho}$ such that $f_{\pm1,\pm1}(g)\neq 0$. Therefore $\rho_2(g)\in\Phi_1(\rho)$ has therefore non-zero $X_{\pm1,\pm1}$-component. Since we have the decomposition
\[\Phi_1(\rho)=\Phi_1(\rho)^{\op{G}_{\{p\}}}\oplus\Bigg(\bigoplus\limits_{(\delta_1,\delta_2)\in D} \Phi_1(\rho)_{\bar{\chi}^{\delta_1\alpha+\delta_2\beta}}\Bigg)\]
all the $X_{\pm1,\pm1}$ are in $\Phi_1(\rho)$.
Note that the cyclotomic character $\chi$ induces an isomorphism
\[\chi:\op{Gal}(\Q(\mu_{p^\infty}/\Q(\mu_p)))\rightarrow 1+p\Z_p.\]
Let $\gamma\in\op{G}_{\Q(\mu_p)}$ be such that $\chi(\gamma)=1+p$. Since $1+p\tilde{X}\in\Phi_1(\rho)$ corresponds to $X\in\op{Ad}^0\bar{\rho}$, the element $\rho_2(\gamma)$ coincides with $w=\alpha(a_{1,1}-a_{3,3})+\beta(a_{2,2}-a_{4,4})$. Hence we have shown that the hypothesis of Lemma \ref{lphi}
are satisfied, concluding the proof.
\end{proof} 

\section{Counting points}\label{seccount}
We now set out to show that there are residual representations for which the hypotheses of Theorem \ref{thmdefo} hold.
All the hypothesis  of the Theorem are conditions on a couple $\pair$ of numbers modulo $(p-1)$, and the set of their linear combinations \[I_{\pair}=\{\delta_1\alpha+\delta_2\beta\;|\;(\delta_1,\delta_2)\in D\}=\{\pm 2\alpha, \pm 2\beta, \pm\alpha\pm\beta\}.\]
We aim to rephrase the hypotheses of Theorem \ref{thmdefo} in terms of $\pair$ and the set $I_{\pair}$. The first hypothesis is that $\alpha+\beta$ is odd, and this does not need rewriting. The second hypothesis can be broken up in two statements, first we require that all the elements of $I_{\pair}$ are distinct, i.e. $\# I_{\pair}=8$, second we ask that $1$ is not in $I_{\pair}$. The third hypothesis can be rephrased as $\mathcal{C}(\bar{\chi}^{p-\epsilon})=0$ for all $\epsilon\in I_{\pair}$.
Thus, in order to guarantee that we can find a pair of exponents $\pair$   such that the deformation problem associated to diagonal representation
\[\bar{\rho}=\begin{pmatrix} \bar{\chi}^\alpha&&&\\ &\bar{\chi}^{\beta} &&\\ &&\bar{\chi}^{-\alpha}&\\ &&&\bar{\chi}^{-\beta} \end{pmatrix}\] is unobstructed we consider the following
\begin{question}\label{Q2}
Under what condition on the prime $p$ can we guarantee the existence of a pair $\pair \in\big(\Z/(p-1)\Z\big)^2$ with $\alpha+\beta$ odd such that
\begin{enumerate}
    \item \#$I_{\pair} =8$, 
    \item $1\not\in I_{\pair} $,
    \item $\mathcal{C}(\bar{\chi}^{p-\epsilon})=0$ for $\epsilon\in I_{\pair}$ ?
\end{enumerate}
\end{question}

    We will show that this question is equivalent to study the number of points of the plane $\big(\Z/(p-1)\Z\big)^2$ whose coordinates have different parities that do not lie on some family of lines. To do so, let $\mathcal{E}=\{\epsilon\in\Z/(p-1)\Z\;|\;\mathcal{C}(\bar\chi^{\epsilon})\neq 0 \}$ be the set of exponents for which eigenspace of eigenvalue $\bar{\chi}^\epsilon$ of the modulo-$p$ class group is non-trivial, and set \[\mathcal{E}^*=\Big\{p-\epsilon\in\Z/(p-1)\Z\,| \epsilon \in \mathcal{E}\Big\}-\Big\{0,\frac{p-1}{2}\Big\}.\]
    \begin{defi}\label{d5}
    We denote by $e$ the cardinality of  $\mathcal{E}^\star$. Note that $e$ is always at most the irregularity index $e_p$ of the prime $p$.
    \end{defi}
Finally, we set $\bar{\mathcal{E}}$=$\mathcal{E}^*\sqcup\{0,1,\frac{p-1}{2}\}$. By how these sets have been defined we have the following
\begin{rmk}\label{rmk5.3}
if $\epsilon$ is in $\mathcal{E}$, then $p-\epsilon$ is in $\bar{\mathcal{E}}$.
\end{rmk}
The first step towards answering Question \ref{Q2} is rephrasing the conditions $(1)-(3)$ in terms of equations modulo $p-1$.

\begin{lem}\label{l:6}
The conditions (1)-(3) of Question \ref{Q2} are equivalent to conditions:
\begin{enumerate}[(a)]
    \item $\pair$ is not a solution to any of the equations $2x=\pm \epsilon$,  $2y=\pm \epsilon$ for $\epsilon\in\bar{\mathcal{E}}$;
    \item $\pair$ is not a solution to any of the equations $y=\pm x \pm \epsilon$ for $\epsilon \in \bar{\mathcal{E}}$;
    \item $\pair$ is not a solution to any of the equations $y=\pm 3x$, $3y=\pm x$.
\end{enumerate}
Where all of the above are in $\Z/(p-1)\Z$.
\end{lem}
\begin{proof}
We start by checking that $(a)-(c)$ follow from $(1)-(3)$.

Condition $(a)$ (resp. $(b)$) is obtained by $(1)-(3)$ by considering the elements $\pm 2\alpha,\pm 2\beta$ (resp $\pm\alpha\pm\beta$) of $I_{\pair}$.  Condition (3) guarantees that $\pair$ is not a solution to the equation of condition $(a)$ (resp. $(b)$) for all $\epsilon$ in $\mathcal{E}^\star$, so we only need to show that $(1)-(3)$ also imply that $\pair$ is not a solution for the aforementioned equation when we set $\epsilon=0,1,\frac{p-1}{2}$. For $\epsilon=1$ this is exactly the content of condition $(2)$. If $\pair$ is a solution of the equation of $(a)$ (resp. of $(b)$), then either $0$ or $\frac{p-1}{2}$ belong to $I_{\pair}$. Since $x \mapsto -x$ is an involution of $I_{\pair}$ that fixes $0,\frac{p-1}{2}$ this would contradict Condition $(1)$.

Condition $(c)$ is similarly granted by Condition $(1)$, as if $\alpha=\pm 3\beta$ (resp $\pm3\alpha =\beta $) we have that $\alpha\mp \beta =\pm 2\beta$ (resp. $\pm2\alpha=\beta\mp\alpha$ ) and $\#I_{\pair} < 8$. 

We are now left to check that $(1)-(3)$ follow from $(a)-(c)$.

As a first step we notice that Condition $(1)$ can be explicitly written as the following set of inequalities:
\[4\alpha \neq 0,4\beta\neq0, \alpha\pm\beta\neq 0, 3\alpha\pm \beta\neq 0, \alpha\pm 3\beta,\neq 0.\]
The first two inequalities follow by condition $(a)$ as $0,\frac{p-1}{2}\in \bar{\mathcal{E}}$ upon noticing that the condition $4z\equiv0$ is equivalent in to $2z=0,\frac{p-1}{2}$ in $\Z/(p-1)\Z$. The following inequality follows directly by $(b)$ as $0 \in\bar{\mathcal{E}}$, and the latest two follow from $(c)$ as again $0$ is in $\bar{\mathcal{E}}$.

Condition $(2)$ is implied by $(a),(b)$ by setting $\epsilon=1\in{\bar{\mathcal{E}}}$.

Condition $(3)$ follows by the same argument and Remark \ref{rmk5.3}.

\end{proof}
\begin{cor}
When we restrict ourselves to consider pairs $\pair$ such that $\alpha+\beta$ is odd, Conditions $(1)-(3)$ are equivalent to Conditions $(a),(b)$.
\end{cor}
\begin{proof}
By Lemma \ref{l:6} that $(1)-(3)$ are equivalent to $(a)-(c)$. What is left to show is that $(c)$ always holds whenever $\alpha+\beta$ is odd. Looking for a contradiction, assume that $(c)$ does not hold meaning that $\pair$ is both such that $\alpha+\beta$ is odd a solution of either $y=\pm3x$ or $x=\pm3y$, reducing modulo $2$ we get $\alpha\equiv\beta$ mod $2$.
\end{proof}
In order to find an answer Question \ref{Q2} we now must count how many points with different parity coordinates of $(\Z/(p-1)\Z)^2$ are cut out by conditions $(a)$ and $(b)$. By imposing that the “number of possible couples" minus the “number of couples cut out by lines defined by our conditions" is be positive we can state the following

\begin{thm}\label{thmpts}
Let $p$ be a prime and $e$ as in Definition \ref{d5}, then, if {$4e+8<\frac{p-1}{2}$}, there exists a pair $\pair$ satisfying conditions $(a), (b)$ such that $\alpha+\beta$ is odd. 
\end{thm}
\begin{proof}
We know that there are $\frac{(p-1)^2}{2}$ points whose coordinates have different parities in the plane $(\Z/(p-1)\Z)^2$. From this we count how many points are cut out by the lines defined Condition $(a)$ and $(b)$ of Lemma \ref{l:6}. Since these depend on $\epsilon\in\bar{\mathcal{E}}=\mathcal{E}^\star\sqcup\Big\{0,1,\frac{p-1}{2}\Big\}$ we first treat the case $\epsilon\in\Big\{0,1,\frac{p-1}{2}\Big\}$ and then the case depending on $e=\#\mathcal{E}^\star$.

We begin by noticing that $\epsilon=0$ defines 2 vertical lines ($x=0$, $x=\frac{p-1}{2}$), and 2 horizontal lines ($y=0$, $y=\frac{p-1}{2}$) on which each other point satisfies condition $(c')$, and 2 oblique lines (as $0=-0$), which are composed of points whose coordinates all have the same parity. Therefore this gives no contributions, for a total of $2(p-1)$ points excluded by this condition.
\\We treat the case $\epsilon=1$ in a similar fashion. Since $2$ is not a unit modulo $p-1$, there are just $4$ lines in this case, namely the lines $y=\pm x \pm 1$ and all of their points have different parity coordinates, excluding a total of $4(p-1)$ points.
\\The case $\epsilon=\frac{p-1}{2}$ behaves like the case $\epsilon=0$ if $p\equiv 1$ mod $4$ but defines no vertical/horizontal lines if $p=3$ mod 4, so that we have either 4 lines with each other point having different parity coordinates, or 2 oblique lines of whose points all have different coordinates, for a total of $2(p-1)$ points.
Adding these together we obtain the number of points that we have to remove independently of the irregularity of $p$ for a total of $8(p-1)$.
In order to compute
 the number of points cut out by lines defined from other elements $\epsilon\in\mathcal{E}^\star\subset\bar{\mathcal{E}}$ we partition $\mathcal{E}^\star=\mathcal{E}^\star_0\sqcup\mathcal{E}^\star_1$ in its set of even and odd elements respectively. We call $\#\mathcal{E}^\star_0=e_0$ and $\#\mathcal{E}^\star_1=e_1$.
we have $2\epsilon\neq 0$ by definition of $\mathcal{E}^\star$, so that condition (a) gives a total of $8$ lines for each $\epsilon\in \mathcal{E}^\star$ in which every other point has coordinates having different parity, and condition (b) gives $4$ more, whose points all have coordinates of same parity, for a total of $4e_0(p-1)$ points.
We now are left to consider the case $\epsilon\in\mathcal{E}^\star_1$. This behaves like the case $\epsilon=1$, the only contribution will be given by oblique lines, as $\epsilon\not\in 2(\Z/(p-1)\Z)$. The $4$ oblique lines $x=\pm y\pm \epsilon$ give us $4(p-1)$ points, as $\epsilon$ being odd implies that all solutions $(x,y)$ have different parities. In total we get $4e_1(p-1)$ points, so that the points cut out by the condition defined by the elements of $\mathcal{E}^\star$ add up to $4(p-1)(e_0+e_1)=4(p-1)e$.

We are left with the inequality:
\begin{equation*}
    \frac{(p-1)^2}{2}-4(p-1)(2+e)>0
\end{equation*}
As we have assumed $4e+8\leq\frac{p-1}{2}$, the inequality above holds, showing there is at least a point $\pair$ with $\alpha+\beta$ odd satisfying $(a),(b)$.
\end{proof}
\begin{cor}
Let $p$ be a prime and $e_p$ its irregularity index, if $4e_p+8<\frac{p-1}{2}$, there exists a residual representation $\bar{\rho}_{\pair}$ satisfying the hypotheses of Theorem \ref{thmdefo}
\end{cor}
\begin{proof}
Follows from Theorem \ref{thmpts}, Lemma \ref{l:6} and Definition \ref{d5}.
\end{proof}
\begin{ex}[Finding $\pair$ when $p=37$]
The prime $37$ has irregularity index 1, as the eigenspace $\mathcal{C}(\bar{\chi}^7)$ is nontrivial. We are still in the hypothesis of Theorem \ref{thmpts} as \[4e_{37}+8=12\leq 18=\frac{37-1}{2}.\]
In order to find a $\pair$ such that $\bar{\rho}_{\pair}$ satisfies the hypothesis of Theorem \ref{thmdefo}, we go through the proof of Theorem \ref{thmpts}.
Drawing a picture of the lines that do not depend on the irregularity in $(\Z/36\Z)^2$ we see that these cut out all the points except the set $S=S_{0,0}\cup S_{1,0}\cup S_{0,1}\cup S_{1,1} $, where
$S_{0,0}=\{(\pm 1,\pm 6),(\pm 1,\pm 7),(\pm 2,\pm 4),(\pm 2,\pm 8),(\pm 3,\pm 8),$\\$(\pm 4,\pm 8),(\pm 4,\pm 2),(\pm 5,\pm 1),(\pm 5,\pm 7),(\pm 6,\pm 1),(\pm 7,\pm 1),(\pm 7,\pm 5),(\pm 8,\pm 2),(\pm 8,\pm 4)\},$\\
and $S_{i,j}=\{(m+18i,n+18j) \text{ where } (m,n)\in S_{0,0}\}.$

A quick glance at a plotting of this set (or of the set of lines) shows that the two couples $(1,6),(12,5)$ cannot lie in the set of lines defined by one exponent only, in particular, as we know $i=7$ is the exponent that realizes the irregularity of 37 we can see that the couple $(12,5)$ works by computing the set  $$p-I_{(12,5)}=\{8,11,13,18,20,25,27,39\}\not\ni 7.$$
Hence we know that the deformation problem associated to the diagonal representation $\bar{\rho}_{(12,5)}:\op{G}_{\{37\}}\rightarrow\operatorname{GSp}_4(\F_{37})$ is unobstructed.
\end{ex}

\bibliographystyle{abbrv}
\bibliography{references}

\begin{thebibliography}{10}

\bibitem{barnet2011family}
T.~Barnet-Lamb, D.~Geraghty, M.~Harris, and R.~Taylor.
\newblock A family of {C}alabi--{Y}au varieties and potential automorphy ii.
\newblock {\em Publications of the Research Institute for Mathematical
  Sciences}, 47(1):29--98, 2011.

\bibitem{breuil2001modularity}
C.~Breuil, B.~Conrad, F.~Diamond, and R.~Taylor.
\newblock On the modularity of elliptic curves over {$\mathbb{Q}$}: wild 3-adic
  exercises.
\newblock {\em Journal of the American Mathematical Society}, pages 843--939,
  2001.

\bibitem{cassels1968algebraic}
J.~Cassels and A.~Frohlich.
\newblock Algebraic number theory, proceedings of the {B}righton conference,
  1968.

\bibitem{cornell2013modular}
G.~Cornell, J.~H. Silverman, and G.~Stevens.
\newblock {\em Modular forms and {F}ermat’s last theorem}.
\newblock Springer Science \& Business Media, 2013.

\bibitem{cornut2018generators}
C.~Cornut and J.~Ray.
\newblock Generators of the pro-p {I}wahori and {G}alois representations.
\newblock {\em International Journal of Number Theory}, 14(01):37--53, 2018.

\bibitem{greenberg2016galois}
R.~Greenberg.
\newblock Galois representations with open image.
\newblock {\em Annales math{\'e}matiques du Qu{\'e}bec}, 40(1):83--119, 2016.

\bibitem{maire2021galois}
C.~Maire.
\newblock On {G}alois representations with large image.
\newblock {\em arXiv preprint arXiv:2104.02505}, 2021.

\bibitem{neukirch2013cohomology}
J.~Neukirch, A.~Schmidt, and K.~Wingberg.
\newblock {\em Cohomology of number fields}, volume 323.
\newblock Springer Science \& Business Media, 2013.

\bibitem{ray2021constructing}
A.~Ray.
\newblock Constructing {G}alois representations ramified at one prime.
\newblock {\em Journal of Number Theory}, 222:168--180, 2021.

\bibitem{ribet1976modular}
K.~A. Ribet.
\newblock A modular construction of unramified p-extensions of {$\mathbb{Q}$}
  ($\mu$p).
\newblock {\em Invent. math}, 34(3):151--162, 1976.

\bibitem{serre2013galois}
J.-P. Serre.
\newblock {\em Galois cohomology}.
\newblock Springer Science \& Business Media, 2013.

\bibitem{tangalgebraic}
S.~Tang.
\newblock Algebraic monodromy groups of {G}-valued {$\ell$}-adic {G}alois
  representations.
\newblock {\em Algebra \& Number Theory, available as}.

\bibitem{taylor1995ring}
R.~Taylor and A.~Wiles.
\newblock Ring-theoretic properties of certain {H}ecke algebras.
\newblock {\em Annals of Mathematics}, pages 553--572, 1995.

\bibitem{washington1997introduction}
L.~C. Washington.
\newblock {\em Introduction to cyclotomic fields}, volume~83.
\newblock Springer Science \& Business Media, 1997.

\bibitem{wiles1995modular}
A.~Wiles.
\newblock Modular elliptic curves and {F}ermat's last theorem.
\newblock {\em Annals of mathematics}, 141(3):443--551, 1995.

\end{thebibliography}

\end{document}